\documentclass{amsart}[12pt,draft]

\usepackage{amsmath,amssymb,amscd,amsthm,indentfirst}
\usepackage{amsfonts,eufrak}
\usepackage[mathscr]{eucal}
\usepackage{cases}
\usepackage[all]{xy}
\usepackage{enumerate}

\usepackage{hyperref}

\swapnumbers
\newtheorem{prop}[equation]{Proposition}
\newtheorem{thm}[equation]{Theorem}
\newtheorem{cor}[equation]{Corollary}
\newtheorem{lem}[equation]{Lemma}

\theoremstyle{definition}
\newtheorem{defn}[equation]{Definition}

\newtheorem{remark}[equation]{Remark}

\newtheorem{exa}[equation]{Example}

\def\A{\ensuremath{\mathcal{A}}}

\def\I{\ensuremath{\mathcal{I}}}

\def\K{\ensuremath{\mathcal{K}}}

\def\P{\ensuremath{\mathcal{P}}}
\def\Q{\ensuremath{\mathcal{Q}}}

\def\NN{\ensuremath{\mathbb{N}}}
\def\ZZ{\ensuremath{\mathbb{Z}}}

\def\id{\textrm{id}}

\def\rank{\operatorname{rank}}

\def\ring{commutative ring with identity}

\newcounter{saveequation}

\date{\today}
\title{A local method for posets}
\author{Antonio D\'{i}az Ramos}
\address{Departamento de {\'A}lgebra, Geometr{\'\i}a y Topolog{\'\i}a,
Universidad de M{\'a}\-la\-ga, Apdo correos 59, 29080 M{\'a}laga,
Spain.}

\thanks{Supported by MICINN grant RYC-2010-05663. Partially supported by MEC grant MTM2013-41768-P and Junta de Andaluc{\'\i}a grant FQM-213.}

\email{adiazramos@uma.es}

\begin{document}

\begin{abstract}
We propose some conditions on a poset that produce a small chain complex for its homology. This allows to compare simplicial complexes and Quillen's complexes under the same prism. It turns out they differ in the existence or not of free faces in an acyclic complex.
\end{abstract}

\maketitle
\section{Introduction}\label{section:Introduction}
In this work we investigate a collection of posets that, roughly speaking, are those posets whose rays are contractible in a homogeneous way. Among them we find the face poset $\P(\Delta)$ of any abstract simplicial complex $\Delta$ and Quillen's complex $\A_p(G)$ of any finite group $G$ at a prime $p$ \cite{Quillen1978}. Recall that the face poset $\P(\Delta)$ is the poset of non-empty simplices of $\Delta$ ordered by inclusion and that $\A_p(G)$ is the poset of non-trivial elementary abelian $p$-subgroups of $G$ ordered by inclusion. In the setup we propose, simplicial complexes appear as the limit case of Quillen's complexes for $p=1$.

In order to introduce these notions, let $\P$ be a poset, $p\in \P$ an object and consider the subposet $\P_{\leq p}=\{r\in \P|r\leq p\}$. This ray $\P_{\leq p}$ is contractible as it contains the terminal object $p$, and hence acyclic. We assume that the ray $\P_{\leq p}$ is finite for all $p\in \P$ and that $\P$ does not contain a minimum element. Then we denote by $\hat\P$ the poset obtained by augmenting $\P$ with a minimum element $\hat0$. Assume there exists a function $\dim:\hat\P\to \ZZ$ such that $\dim p=\dim q+1$ if $p$ covers $q$, and with $\dim\hat 0=-1$. We call a poset equipped with such a function a \emph{graded} poset. We shall say that a collection $\K=\{\K_p,\eta_p\}_{p\in \P}$, where $\K_p$ is a subposet of $\hat\P_{\leq p}$ containing $\hat 0$ and $\eta_p\colon \K_p\rightarrow \hat\P_{\leq p}\setminus \K_p$ is a map, is a \emph{local covering family} for $\P$ if:
\[
\text{for all $p\in \P$ and all $r\in \K_p$: $r<\eta_p(r)$ and $\dim
 \eta_p(r)=\dim r+1$,}
\]
and the maps $\{\eta_p\}_{p\in \P}$ are compatible in a certain way (see Definition \ref{definitionlocalcoveringfamily} for full details). These conditions are similar to those of Stanley on decomposition of acyclic simplicial complexes \cite[Theorem 1.2]{Stanley1993} and of Forman on Discrete Morse Theory for cell complexes \cite{Forman1998}. They differ in that there the maps $\eta_p$ are assumed to be bijective, a condition that is not imposed here.

For instance, let $\Delta$ be a (possibly infinite) simplicial complex and choose a total order on its vertices. Consider its face poset $\P(\Delta)$ and, for $\sigma \in \P(\Delta)$, set $\dim(\sigma)=|\sigma|-1$ and define $\K_\sigma$ and $\eta_\sigma$ by
\[
\text{$\K_\sigma=\{\tau\subseteq \sigma| \sigma^*\notin \tau\}$ and  $\eta_\sigma(\tau)=\tau\cup \{\sigma^*\}$},
\]
where $\sigma^*=\min \sigma$. Then $\K_\Delta=\{\K_\sigma,\eta_\sigma\}_{\sigma\in \Delta}$ is a local covering family for $\P(\Delta)$. In this case the maps $\eta_\sigma$ are bijections and $\hat 0$ corresponds to the empty simplex.

The situation for Quillen's complex $\A_p(G)$ of the finite group $G$ at the prime $p$ is similar. Start by choosing a total order for the order-$p$ subgroups of $G$. Then we can define a local covering family $\K_{\A_p(G)}$ for $\A_p(G)$ by setting, for $H\in \A_p(G)$, $\dim(H)=\rank(H)-1$ and the following:
\[
\text{$\K_H=\{I\leq H|H^*\nleq I\}$ and $\eta_H(I)=\langle I,H^*\rangle$},
\]
where $H^*=\min\{V|\text{$\rank(V)=1$ and $V\leq H$}\}$. Note that in this case the maps $\eta_H$ are  surjections and $\hat 0$ corresponds to the trivial subgroup.

It is not hard to see (see below) that similar local covering families can be constructed on any poset $\P$ for which for all $p\in \P$ we have that $\P_{\leq p}$ is isomorphic to either the face poset of a $(\dim p)$-simplex or to $\A_p(C_p^{\dim p+1})$. We call such posets \emph{locally simplicial} posets or \emph{locally $p$-Quillen} posets respectively.

Next, we further discuss the use of local covering families. Throughout this paper $R$ is a \ring . Let $\Delta$ be a simplicial complex. We denote by $C_*(\Delta;R)$ the simplicial chain complex of $\Delta$ with coefficients in $R$.  We also denote by $|\Delta|$ the topological realization of $\Delta$ and, for a poset $\P$, we define its realization by $|\P|=|\Delta(\P)|$, where $\Delta(\P)$ denotes the order complex of the poset $\P$. In particular, if $\P$ is a poset we compute the homology of its realization via  $H_*(|\P|;R)\cong H_*(C_*(\Delta(\P);R))$.

For the face poset $\P= \P(\Delta)$ of the simplicial complex $\Delta$, we have 
\[
H_*(|\P|;R)\cong H_*(C_*(\Delta(\P);R))\cong H_*(C_*(\Delta;R))
\]
and so we can compute $H_*(|\P|;R)$ using the smaller chain complex $C_*(\Delta;R)$ instead of the larger $C_*(\Delta(\P);R)$. This is not true in general as there is no such thing as $C_*(\P;R)$ for arbitrary $\P$. What we prove here is that if the poset $\P$ is equipped with a local covering family $\K$ then there is a chain complex $C_*^\K(\P;R)$ that does play the right role, i.e., it removes a subdivision. 

\begin{thm}\label{thm_introd_cc}
Let $\P$ be a graded poset with local covering family $\K$ and let $R$ be a \ring. Then there is a chain complex $C^\K_*(\P;R)$ with 
\[
C_n^\K(P;R)=\bigoplus_{p\in \P_n} R^{K^p_n}
\]
whose homology is $H_*(|\P|;R)$.
\end{thm}
Here $\P_n=\{r\in \P|\dim r=n\}$, the objects of dimension $n$. The numbers $K^p_*$ depend upon $\K$ and are defined inductively by $K^p_0=1$ and $K^p_{n+1}=\sum_{q\in (\K_p )_{n}} K^q_n$. A detailed description of the differential is given in Section \ref{section:Explicit differential}. Next we investigate the geometric meaning of the local covering families $\K_\Delta$ and $\K_{A_p(G)}$ defined above. In the former case, for any simplex $\sigma$ of the simplicial complex $\Delta$ we have $K^\sigma_{\dim(\sigma)}=1$. The geometric meaning of $\K_\Delta$ is explained by the next result.

\begin{thm}\label{thm_introd_ccsimp}
Let $\Delta$ be a (possibly infinite) simplicial complex and let $\K_\Delta$ be the local covering family as above. Then
\[
C_*^{\K_\Delta}(\P(\Delta);R)\cong C_*(\Delta;R).
\]
\end{thm}

This means $\K_\Delta$ give us back the usual simplicial chain complex of $\Delta$. For the covering family constructed for $\A_p(G)$, where $G$ is a finite group and $p$ is a prime, we have $K^H_n=p^\frac{n(n+1)}{2}$ for a subgroup $H\cong C_p^{n+1}$. Theorem \ref{thm_introd_cc} applied to $\K_{A_p(G)}$ gives the following:

\begin{thm}\label{thm_introd_ccapg}
Let $G$ be a finite group and let $p$ be a prime. Then 
there is a chain complex with $n$-chains $\bigoplus_{H\in \A_p(G)_n} R^{p^\frac{n(n+1)}{2}}$ and with explicit differential whose homology is $H_*(|\A_p(G)|;R)$. In particular, 
\[
\chi(|\A_p(G)|)=\sum_{n=0}^{rk_p(G)-1} (-1)^np^\frac{n(n+1)}{2}|\A_p(G)_n|.
\]
\end{thm}
Here, $rk_p(G)$ is the $p$-rank of $G$, i.e., the largest dimension of an elementary abelian $p$-subgroup of $G$. 
Theorems \ref{thm_introd_ccsimp} and \ref{thm_introd_ccapg} are clearly independent of the total orders chosen to define the local covering families $\K_\Delta$ and $\K_{\A_p(G)}$ respectively. The differential of the chain complex in Theorem \ref{thm_introd_ccapg} behaves ``locally" as the simplicial differential (see Example \ref{exampleQcdifferential}). In order to compare the two local covering families introduced so far, $\K_\Delta$ and $\K_{\A_p(G)}$, note first that the formulas for $K_n^H$ and $K^\sigma_n$ coincide if we let $p=1$. Nevertheless, there are qualitative differences: call an object of the poset $\P$ a \emph{free object} if it is  a non-maximal object which is smaller than a unique maximal object. So for simplicial complexes this is the usual concept of free face. It is well known that there exist simplicial complexes which are contractible but nevertheless they do no have a free face, e.g., any triangulation of the dunce hat. The situation for Quillen's complex, or more generally for locally $p$-Quillen posets, is quite the opposite. We define the dimension of $\P$ by $\dim \P=\max_{p\in \P} \dim p$ if the maximum exists and by $\dim \P=\infty$ otherwise.

%

\begin{thm}\label{thm_intro_freefaceonp}
Let $\P$ be a locally $p$-Quillen finite poset at the prime $p$ with $\dim \P=N$. If the homology group $\widetilde H_N(|\P|;R)$ is zero then the proportion $r$ of free objects among the non-maximal objects of $\P_{N-1}$ satisfies:
\[
r\geq \frac{p^{N+1}-2p^N+1}{p^{N+1}-p^N}>0.
\]
\end{thm}

In particular, contractible locally $p$-Quillen posets always have a free object. Furthermore, the ratio $r$ in the statement tends to $1^-$ as $p\to \infty$, i.e., asymptotically on the prime $p$, a locally $p$-Quillen poset $\P$ of dimension $N$ satisfying $H_N(|\P|;R)=0$ can be collapsed into $\P_{\leq N-1}=\{r\in \P|\dim r\leq N-1\}$. This result seems to pave the way for an asymptotic approach to Quillen's conjecture \cite{Quillen1978} on the poset $\A_p(G)$. 

Note that there exist posets that admit a local covering family but are not shellable: In  \cite{Shareshian2004}, Shareshian showed that there are finite groups $G$ such that the homology of $\A_p(G)$ has torsion. This implies that $\A_p(G)$ cannot be shellable. Nevertheless, $\A_p(G)$ can always be equipped with a local covering family as above. Also, there are posets that may be equipped with a local covering family but are not Cohen-Macaulay: any non-Cohen-Macaulay simplicial complex is an example.

In general, a poset $\P$ will admit a local covering family if it is locally atom-modular, i.e., if all the subposets $\{\P_{\leq p}\}_{p\in \P}$ are \emph{atom-modular}. Here, we say that a graded poset $\Q$ is  \emph{atom-modular} if for every atom $a\in \Q$ and every $q\in \Q$ we have either $a\leq q$ or the least upper bound $a\vee q$ exists and satisfies $\dim(a\vee q)=\dim(q)+1$. For instance, if $\hat Q$ is either a graded semimodular lattice or a Boolean lattice then $\Q$ is atom-modular. This covers the lattice of finite subsets of a set and finite-dimensional subspaces of a vector space. To define a local covering family on a locally atom-modular poset $\P$, choose a well ordering for its atoms and define $p^*=\min\{a\leq p|\dim a=0\}$. Then set
\[
\text{$\K_p=\{r\leq p|p^*\nleq r\}$ and $\eta_p(r)=p^*\vee r$}.
\]
This construction applies to the aforementioned locally simplicial posets and locally $p$-Quillen posets and coincides with the description given for simplicial complexes and Quillen's complexes of finite groups.

\textbf{Organization of the paper:} In Section \ref{section:wedgicalposets} we introduce preliminary notions, including a (folklore) chain complex for a graded poset, and certain notions of ``suspension" and ``truncation". This is followed in Section \ref{section:Localcoveringfamilies} by the definition of local covering family and proof of Theorems \ref{thm_introd_cc} and \ref{thm_introd_ccapg}. The study of the differential is postponed until Section \ref{section:Explicit differential}, where Theorem \ref{thm_introd_ccsimp} is also proven. The treatment of free objects is carried out in Section \ref{section:freeobjects}.

\textbf{Acknowledgements:} I am in debt to Torsten Ekedahl, who sadly passed away in 2011. He improved the definition of local covering family, removing some integer equalities from the original notion, and also provided a more geometric proof of Theorem \ref{thm_introd_cc}. The terminology ``atom-modular" is due to him.

\section{Spherical posets}\label{section:wedgicalposets}
We only consider posets $\P$ that do not have a smallest element and such that $\P_{\leq p}:=\{r\in\P|r\leq p\}$ is finite for all $p\in \P$. We also use the notation $\P_{<p}:=\{r\in\P|r<p\}$, and, for an integer $n$, we write $\P_{\leq n}:=\{r\in\P|\dim r \leq n\}$  and $\P_{n}:=\{r\in\P|\dim r =n\}$. A grading on a poset $\P$ is a a function $\dim\colon\P\rightarrow \NN$ such that $p<q\implies\dim p<\dim q$ and such that $\dim$ takes the value $0$ on minimal elements. We also assume that $\dim p=\dim q+1$ if $p$ covers $q$. A poset with a grading is termed a \emph{graded} poset, and every poset we consider is graded unless stated otherwise. We define the \emph{dimension} of the poset $\P$ by $\dim \P=\max_{p\in \P} \dim p$ if the maximum exists and by $\dim \P=\infty$ otherwise.
If $\P$ is a poset we denote by $\hat\P$ the poset obtained by augmenting $\P$ with a minimum element $\hat0$ for which $\dim\hat0=-1$ (and thus $\hat\P$ is formally speaking not a poset with a grading). Using an expression $\hat\P$ automatically means that $\hat\P$ is the augmentation of a graded poset $\P$.

Following Quillen \cite{Quillen1978}, we say that a poset $\P$ of dimension $d$ is \emph{$d$-spherical} if $|\P|$ is $(d-1)$-connected, or equivalently, if it has the homotopy type of a bouquet of $d$-spheres. We say that the poset $\P$ is \emph{locally spherical} if $\P_{<p}$ is $(\dim p-1)$-spherical for each $p\in \P$. Here we define a bouquet of $(-1)$-spheres as the empty set $\emptyset$.

Now let $\P$ be a graded poset, let $R$ be a \ring{ }and let $C_*(\Delta(\P);R)$ the chain $R$-complex that computes the (simplicial) reduced homology of the realization $|P|=|\Delta(\P)|$ of $\P$ with coefficients in $R$, $\widetilde H_*(|\P|;R)$. The $n$-chains are given by
\[
C_n(\Delta(\P);R)=\bigoplus_{p_0<\ldots<p_n} R
\]
for $n\geq 0$ and by $C_{-1}(\Delta(\P);R)=R$. It is equipped with the usual differential $\partial=\sum_{i=0}^n (-1)^i\partial_i$ and the usual augmentation map $\epsilon$. Here, for $0\leq i\leq n$, the $R$-linear map $\partial_i\colon C_n(\Delta(\P);R)\to C_{n-1}(\Delta(\P);R)$ sends
\[
p_0<\ldots<p_n\mapsto p_0<\ldots<p_{i-1}<p_{i+1}<\ldots<p_n,
\]
and the augmentation map sends $p_0\in \P$ to $\epsilon(p_0)=1\in R$. The filtration of spaces 
\[
\emptyset\subseteq |\P_{\leq 0}|\subseteq\ldots\subseteq |\P_{\leq i}|\subseteq |\P_{\leq i+1}|\subseteq\ldots\subseteq |\P|
\]
gives rise to the following increasing filtration of this chain complex
\[
F_iC_j(\Delta(\P);R)=\bigoplus_{p_0<\ldots<p_j\textit{, $\dim p_j\leq  i$}} R
\]
for $j\geq 0$, and $F_iC_{-1}(\Delta(\P);R)=R$ for $i\geq -1$ and $0$ otherwise. This filtration is exhaustive and bounded below and hence we have the following (folklore) convergent homological-type spectral sequence
\[
E^1_{i,j}=\widetilde H_{i+j}(F_iC_{i+j}/F_{i-1}C_{i+j}) \Rightarrow \widetilde H_{i+j}(|\P|;R).
\]
Notice that $F_iC_{i+j}/F_{i-1}C_{i+j}=\bigoplus_{p_0<\ldots<p_{i+j}\textit{, $\dim p_{i+j}=i$}} R$ which is $0$ unless $j\leq 0$. Moreover, we have 
\[
\widetilde H_{i+j}(F_iC_{i+j}/F_{i-1}C_{i+j})\cong
\widetilde H_{i+j}(\bigvee_{\dim p=i} \Sigma|P_{<p}|;R)\cong
\bigoplus_{\dim p=i}  \widetilde H_{i+j-1}(|P_{<p}|;R)
\]
if $j\leq 0$ and $0$ otherwise, where we define for convenience $\widetilde H_{-1}(\emptyset;R)=R$. So, if $\P$ is locally spherical, the spectral sequence degenerates to the chain complex
\begin{equation}\label{chaincomplexforlocallyspherical}
 \ldots\to\bigoplus_{\dim p=i}\widetilde H_{i-1}(|\P_{<p}|;R) \rightarrow \ldots\rightarrow 
\bigoplus_{\dim p=1} \widetilde H_0(|\P_{<p}|;R)\rightarrow
\bigoplus_{\dim p=0} R\rightarrow R.
\end{equation}
For future use we introduce, for a fixed object $p\in \P$ of dimension $n$, the linear maps ``suspension at $p$'', $s_p$, and ``truncation at $p$'', $t_p$:
\[
\xymatrix{
C_{n-1}(\Delta(P_{<p});R)\ar@(ur,ul)[r]^{s_p}& 
F_nC_n(\Delta(\P);R).\ar@(dl,dr)[l]_{t_p}\\
}
\]
They are defined on basic elements by
\begin{align*}
s_p(p_0<\ldots<p_{n-1})&=p_0<\ldots<p_{n-1}<p\textit{ and}\\
t_p(p_0<\ldots<p_{n-1}<p_n)&=\begin{cases} p_0<\ldots<p_{n-1} & \textit{ if $p_n=p$,}\\
0 & \textit{ otherwise,}
\end{cases}
\end{align*}
where $\dim p_i=i$. They posses the following properties:
\begin{align}
\partial(s_p(z))&=s_p(\partial(z))+(-1)^nz\textit{, }\label{equationsptppropertiesa}\\
\partial_i(t_p(z))&=t_p(\partial_i(z))\textit{ for $i=0,\ldots,n-1$, and}\label{equationsptppropertiesb}\\
t_{p'}s_p&=\begin{cases}
\id & \textit{ if $p'=p$,}\\
0   & \textit{ otherwise,}\label{equationsptppropertiesc}
\end{cases}
\end{align}
where $p'$ is any object of dimension $n$. Employing Equations \eqref{equationsptppropertiesa} and \eqref{equationsptppropertiesb}, it is immediate that  suspension and truncation induce the inverse to each other isomorphisms:
\begin{equation}\label{isomorphismfromfilteredetosubposet}
\xymatrix{
\bigoplus_{\dim p=n} \widetilde H_{n-1}(|\P_{<p}|;R)\ar@(ur,ul)[r]^{\sum_{\dim p=n}[s_p]}
&
\widetilde H_n(F_nC_n/F_{n-1}C_n).\ar@(dl,dr)[l]_{\oplus_{\dim p=n}[t_p]}
}
\end{equation}

The chain complex (\ref{chaincomplexforlocallyspherical}) has homology equal to $\widetilde H_*(|\P|;R)$ and its differential $d$
\begin{equation}\label{differentialchaincomplexforlocallyspherical}
\widetilde H_n(F_nC_n/F_{n-1}C_n)\stackrel{d}\longrightarrow \widetilde H_{n-1}(F_{n-1}C_{n-1}/F_{n-2}C_{n-1})
\end{equation}
is induced by the differential $\partial$ of $C_*(\Delta(\P);R)$. 
Hence, the differential of the homology class $[z]\in \widetilde H_n(F_nC_n/F_{n-1}C_n)$ is  $d([z])=[(-1)^n\partial_n(z)]$.


\section{Local covering families}\label{section:Localcoveringfamilies}

Forman's Morse theory \cite{Forman1998} for cell complexes runs parallel to classical Morse Theory. For example, the homotopy type of a simplicial complex is determined by the unmatched simplices in an acyclic matching of the Hasse diagram of the complex \cite[Proposition 3.3]{Chari2000}. A partial converse to this result was proven by Stanley \cite[Theorem 1.2]{Stanley1993}. Here we describe similar notions for a poset that determine the homotype type of the realization $|\P_{<p}|$ for each $p\in \P$.
\begin{defn}
Let $\P$ be a graded poset. A local covering family $\K$ for $\P$ is a family of subposets $\K_p\subseteq \hat\P_{\leq p}$  and maps $\eta_p:\K_p\rightarrow \hat\P_{\leq p}\setminus \K_p$ for all $p\in \P$ such that:
\begin{enumerate}\label{definitionlocalcoveringfamily}
\item the element $\hat 0$ belongs to $\K_p$, \label{definitionlocalcoveringfamilyhat0belongs}
\item if $q\in \K_p$ then $q<\eta_p(q)$ and $\dim \eta_p(q)=\dim q+1$, 
\label{definitionlocalcoveringfamilyeta+1}
\item if $q\notin \K_p$ and $q\leq p$ then $\eta_p(\hat 0)=\eta_q(\hat 0)$,
\label{definitionlocalcoveringfamilyhat0}
\item if $q\in \K_p$, $r\in \K_p$ and $q\leq r$ then $\eta_p(q)\leq \eta_p(r)$, and
\label{definitionlocalcoveringfamilyposetmap1}
\item if $q\in \K_p$, $r\notin \K_p$ and $q\leq r\leq p$ then $\eta_p(q)\leq r$.
\label{definitionlocalcoveringfamilyposetmap2}
\end{enumerate}
\end{defn}

\begin{remark}\label{remarklcfdegree-1and0}
Note that condition \eqref{definitionlocalcoveringfamilyhat0belongs} implies that $(\hat \P_{\leq p})_{-1}=(\K_p)_{-1}=\{\hat 0\}$ and then condition \eqref{definitionlocalcoveringfamilyhat0} gives that $(\P_{\leq p})_0=(\K_p)_0\cup \{\eta_p(\hat 0)\}$.
\end{remark}


\begin{exa}\label{examplesccoveringfamilyK}
Let $\P(\Delta)$ be the face poset of a (possibly infinite) simplicial complex $\Delta$ graded by simplicial dimension and choose a total order on its vertices. We  equip $\P(\Delta)$ with a local covering family $\K_\Delta$ as follows: For the $n$-simplex $\sigma$ set $\sigma^*=\min \sigma$ and define the subposet $\K_\sigma\subseteq \hat\Delta_{<\sigma}$ by 
\[
\K_\sigma=\{\textit{subsets $\tau$ of $\sigma$ such that $\sigma^*\notin \tau$}\}.
\]
The map $\eta_\sigma$ sends $\tau\in \K_\sigma$ to $\eta_p(\tau)=\tau\cup\{\sigma^*\}$. It is straightforward that the axioms in Definition \ref{definitionlocalcoveringfamily} are satisfied.
\end{exa}
\begin{exa}\label{exampleQccoveringfamilyK}
Quillen's complex $\A_p(G)$ of a finite group $G$ may be graded by $\dim H=rank(H)-1$ and can also be equipped with a local covering family $\K_{\A_p(G)}$: First, give a total order to the order $p$ subgroups of $G$. Then, for $H\in \A_p(G)$, set $H^*=\min\{V|\text{$\dim(V)=0$ and $V\leq H$\}}$ and define 
\[
\K_H=\{\textit{subgroups $I\leq H$ such that $H^*\nleq I$}\}.
\] 
The map $\eta_H$ maps $I\in \K_H$ to the subgroup $\eta_H(I)$ of $H$ generated by $I$ and $H^*$. It is easy to check that Definition \ref{definitionlocalcoveringfamily} is fulfilled.
\end{exa}

\begin{lem}\label{lemmaposetwithlcfislocallyspherical}
Let $\P$ be a graded poset with a local covering family $\K$. Then $\P$ is locally spherical. For $p\in P$ with $\dim P=n$ the number of $(n-1)$-spheres in the bouquet $|\P_{<p}|$ is equal to the number of $n$-simplices $p_0<p_1<\ldots<p_n=p$ with $\dim p_i=i$ and $p_i\in \K_{p_{i+1}}$ for $i=0,\ldots,n-1$.
\end{lem}
\begin{proof}
We prove by induction on $n=\dim p$ that $|P_{<p}|$ is $(n-1)$-spherical with the given number of spheres. For $n=0$, the claim is clear. Let $p$ be of dimension $n\geq 1$ and consider the subposet given by $\I=\P_{<p}\setminus (\K_p)_{n-1}$. Then $\I$ is not empty because it contains $\eta_p(\hat 0)$. We show that $\I$ is conically contractible \cite[1.5]{Quillen1978} by exhibiting a map of posets $f:\I\rightarrow \I$ such that $x\leq f(x)\geq \eta_p(\hat 0)$ for all $x$ in $\I$. For this is enough to define $f(x)=x$ if $x<p$, $x\notin \K_p$ and $f(x)=\eta_p(x)$ if $x\in \K_p$. This map $f$ is a map of posets because $\K_p$ is a subposet and because of Definition \ref{definitionlocalcoveringfamily}(\ref{definitionlocalcoveringfamilyposetmap1}) and \ref{definitionlocalcoveringfamily}(\ref{definitionlocalcoveringfamilyposetmap2}). Now, by construction, the image of $f$ is contained in $P_{<p}\setminus \K_p$ and hence by Definition \ref{definitionlocalcoveringfamily}(\ref{definitionlocalcoveringfamilyhat0}) we have that $f(x)\geq \eta_{f(x)}(\hat 0)=\eta_p(\hat 0)$ for all $x$ in $\I$.

As $|\I|$ is contractible, the quotient map $|\P_{<p}|\rightarrow |\P_{<p}|/|\I|\simeq \bigvee_{q\in (\K_p)_{n-1}} \Sigma|\P_{<q}|$ is a homotopy equivalence. By induction, for each $q\in (\K_p)_{n-1}$, the realization $|P_{<q}|$ is a bouquet of $(n-2)$-spheres. Hence, $|\P_{<p}|$ is a bouquet of $(n-1)$-spheres and the counting formula for the number of $(n-1)$-spheres in the bouquet holds.
\end{proof}
\begin{cor}\label{cor:localhomologywithrank}
Let $\P$ be a graded poset with a local covering family $\K$ and let $R$ be a \ring. Define inductively the numbers $K^p_n$ for all $p\in \P$ and all $n\in \NN$ as follows: $K^p_0=1$ and $K^p_{n+1}=\sum_{q\in (\K_p )_{n}} K^q_n$. Then for all $p$ in $\P$ the reduced homology group $\widetilde H_{\dim p-1}(|\P_{<p}|;R)$ is free of rank $K^p_{\dim p}$.
\end{cor}

\begin{remark}\label{remark:sizeofKp1}
Note that, for an object $p\in \P$ with $\dim p=1$, we get $K^p_1=|(\K_p)_0|$ and this number, by Remark \ref{remarklcfdegree-1and0}, equals $|(\P_{\leq p})_0|-1$.
\end{remark}

Combining Equation \eqref{chaincomplexforlocallyspherical} with Corollary \ref{cor:localhomologywithrank} we obtain the following result, which is Theorem \ref{thm_introd_cc} of the Introduction.

\begin{thm}\label{chaincomplexforposetwithlcf}
Let $\P$ be a graded poset which has a local covering family $\K$ and let $R$ be a \ring. Then there is a chain complex $C^\K_*(\P;R)$
\[
\ldots\rightarrow\bigoplus_{p\in \P_i} R^{K^p_i}\rightarrow \ldots\rightarrow\bigoplus_{p\in \P_1} R^{K^p_1}\rightarrow\bigoplus_{p\in \P_0} R\rightarrow 0
\]
whose homology is $H_*(|\P|;R)$.
\end{thm}

Note that this chain complex does not depend (up to isomorphism) on the local covering family $\K$ as it is isomorphic to the chain complex (\ref{chaincomplexforlocallyspherical}) coming from the spectral sequence deduced from the grading of $\P$.
\begin{exa}\label{examplescnumberK}
Consider a (possibly infinite) simplicial complex $\Delta$ of finite dimension and the local covering family $\K_\Delta$  for the face poset $\P(\Delta)$ introduced in Example \ref{examplesccoveringfamilyK}. If $\sigma$ is of dimension $n$ we have $|(\K_\sigma)_{n-1}|=1$ and $K^\sigma_n=1$. These computations together with Theorem \ref{chaincomplexforposetwithlcf} applied to $\P(\Delta)$ and $\K_\Delta$ show that in Theorem \ref{thm_introd_ccsimp} of the Introduction we have an isomorphism of $R$-modules on each degree. We shall see in Section \ref{section:Explicit differential} that this isomorphism preserves also the differential, i.e., it is an isomorphism of chain complexes as claimed.
\end{exa}
\begin{exa}\label{exampleQcnumberK}
For Quillen's complex $\A_p(G)$ of a finite group $G$ at a prime $p$ consider the local covering family of Example \ref{exampleQccoveringfamilyK}. For a subgroup $H\cong C_p^{n+1}$ of dimension $n$ it is easy to check that $|(\K_H)_{n-1}|=p^n$ and $K^H_n=p^\frac{n(n+1)}{2}$. From these calculations and Theorem \ref{chaincomplexforposetwithlcf} applied to $\A_p(G)$ and $\K_{\A_p(G)}$ we obtain Theorem \ref{thm_introd_ccapg} of the Introduction.
\end{exa}

\section{Explicit differential}\label{section:Explicit differential}

In this section we study in detail the differential of the chain complex of Theorem \ref{chaincomplexforposetwithlcf}. Fix a graded poset $\P$ with local covering family $\K$ and let $R$ be a \ring. Consider the chain complex $C_*^\K(\P;R)$ of  Theorem \ref{chaincomplexforposetwithlcf} applied to $\P$ and $\K$. If $q\in \P$ and $n=\dim q$ we have, from Equations \eqref{isomorphismfromfilteredetosubposet} and \eqref{differentialchaincomplexforlocallyspherical}, the following commutative diagram:
\[
\xymatrix@=15pt{
\widetilde H_n(F_nC_n/F_{n-1}C_n)\ar[r]^<<<<<<{d}
&\widetilde H_{n-1}(F_{n-1}C_{n-1}/F_{n-2}C_{n-1})\\
\bigoplus_{\dim p=n} \widetilde H_{n-1}(|\P_{<p}|;R)\ar[r]\ar[u]^{\cong}&
\bigoplus_{\dim p=n-1} \widetilde H_{n-2}(|\P_{<p}|;R)\ar[u]^{\cong}\\
\widetilde H_{n-1}(|\P_{<q}|;R)\ar[r]^<<<<<<<<<{d^\P_q}\ar@{^(->}[u]&
\bigoplus_{p\in (\P_{<q})_{n-1}} \widetilde H_{n-2}(|\P_{<p}|;R).\ar@{^(->}[u]
}
\]
It is a routine computation that the map at the bottom is 
\begin{equation}\label{equ:dPq}
d^\P_q=\oplus_{p\in (\P_{<q})_{n-1}} (-1)^n[t_p].
\end{equation}
If we project from the codomain of $d^\P_q$ onto the components with $p\in (\K_q)_{n-1}$ we get another map that we denote by $d^\K_q$:
\begin{equation}\label{equ:dKq}
\xymatrix@=15pt{
\widetilde H_{n-1}(|\P_{<q}|;R)\ar[r]^<<<<<{d^\K_q}&
\bigoplus_{p\in (\K_q)_{n-1}} \widetilde H_{n-2}(|\P_{<p}|;R)\\
[z]\ar@{|->}[r]&d^\K_q([z])=\oplus_{p\in (\K_q)_{n-1}} (-1)^n[t_p(z)].
}
\end{equation}
This map coincides up to sign with the map in homology induced by the homotopy equivalence
\begin{equation}\label{equ:homtopyeqivalence}
|\P_{<q}|\rightarrow |\P_{<q}|/|\I|\simeq \bigvee_{p\in (\K_q)_{n-1}} \Sigma|\P_{<p}|,
\end{equation}
where $\I=\P_{<q}\setminus (\K_q)_{n-1}$. Hence, $d^\K_q$ is an isomorphism. This homotopy equivalence was already used in the proof of Lemma \ref{lemmaposetwithlcfislocallyspherical}.  We record this fact:

\begin{lem}\label{lemma:differentialKisiso}
Let $\P$ be a graded poset with local covering family $\K$ and let $R$ be a \ring. For each $q\in \P$, the map $d^\K_q$ described above is an isomorphism.
\end{lem}

Furthermore, we show below that appropriate basis on each dimension for the chain complex $C_*^\K(\P;R)$ can be chosen so that the description of the isomorphisms $d^\K_q$ is very simple.
\begin{prop}\label{prop:differentialbasisinduction}
Let $\P$ be a graded poset with local covering family $\K$ and let $R$ be a \ring. There are explicit basis of $\widetilde H_{\dim q-1}(|\P_{<q}|;R)$ for all $q\in \P$ such that all isomorphisms $d^\K_q$ carry basic elements to basic elements.
\end{prop}
\begin{proof}
We inductively build a basis of $\widetilde H_{n-1}(|\P_{<q}|;R)$ using the isomorphism $d^\K_q$. For $\dim q=0$ we define $\emptyset_q$ as a generator of $\widetilde H_{-1}(\emptyset;R)=R=R\emptyset_q$. Now assume a basis of $\widetilde H_{n-2}(|\P_{<p}|;R)$ has been constructed for every element $p$ with $\dim p\leq n-1$. We want to construct a basis of $\widetilde H_{n-1}(|\P_{<q}|;R)$ for $q$ with $\dim q=n$ in such a way that $d^\K_q$ carries basic elements to basic elements. This is the homological counterpart to constructing an explicit homotopy inverse to the homotopy equivalence \eqref{equ:homtopyeqivalence}, i.e., to suspend classes from $|\P_{<p}|$.

So fix a basic element $b\in \widetilde H_{n-2}(|\P_{<p}|;R)$ for some $p\in (\K_q)_{n-1}$. We want $B\in \widetilde H_{n-1}(|\P_{<q}|;R)$ with $d^\K_q(B)=b$. Write $b=[z]$ and $B=[Z]$ with $z\in C_{n-2}(\Delta(\P_{<p});R)$ and $Z\in C_{n-1}(\Delta(\P_{<q});R)$. Then it is enough to find $Z$ such that $\partial(Z)=0$ and $d^\K_q([Z])=[z]$ for $z$ with $\partial(z)=0$. Here, $\partial$ is the differential of the complex $C_*(\Delta(\P);R)$ restricted appropriately and $d^\K_q$ is described in \eqref{equ:dKq}. We shall decompose $Z=Z_1+Z_2$, where $Z_i$ correspond to a ``cone over z'' for $i=1,2$, giving a ``suspension of z''.

Set $Z_1$ as the ``suspension at $p$'', i.e., $Z_1=(-1)^ns_p(z)$, where we add the sign for convenience. In order to define $Z_2$, note first that $z\in C_{n-2}(\Delta(\I);R)$, where $\I=\P_{<q}\setminus (\K_q)_{n-1}$, and that $|\I|$ is contractible by the proof of Lemma \ref{lemmaposetwithlcfislocallyspherical}. Hence, there exists some element $Z_2\in C_{n-1}(\Delta(\I);R)$ whose differential is $\partial(Z_2)=z$. Now we have
\[
\partial(Z)=\partial((-1)^ns_p(z))+\partial(Z_2)=(-1)^n(s_p(\partial(z))+(-1)^{n-1}z)+z=0
\]
because of Equation \eqref{equationsptppropertiesa} and because $\partial(z)=0$. Moreover, 
\[
d^\K_q([Z])=\oplus_{p'\in (\K_q)_{n-1}} (-1)^n[(-1)^nt_{p'}(s_p(z))+t_{p'}(Z_2)]=[z],
\]
by Equations \eqref{equ:dKq} and \eqref{equationsptppropertiesc} and because $\I\cap (\K_q)_{n-1}=\emptyset$. If we choose $Z'_2$ with $\partial(Z'_2)=z$,  then $Z'_2-Z_2$ is a boundary and $[Z_1+Z'_2]=[Z_1+Z_2]$. Hence the class $[Z]$ is well defined.
\end{proof}

\begin{remark}\label{rmk:differentialbasisinductiondimq=1}
For $\dim q=1$, the inductive step gives as basis of $\widetilde H_0(|P_{<q}|;R)\cong R^{K^q_1}$ the elements $\{\eta_q(\hat 0)-p\}_{p\in(\K_q)_0}$, see Remarks \ref{remarklcfdegree-1and0} and \ref{remark:sizeofKp1}. 
\end{remark}



\begin{exa}\label{examplescdifferential}
Consider a (possibly infinite) simplicial complex $\Delta$ of finite dimension and the local covering family $\K_\Delta$ for the face poset $\P(\Delta)$ described in Example \ref{examplesccoveringfamilyK}. Denote by $\prec$ the total order chosen on the vertices of $\Delta$ and by $b_\sigma=[z_\sigma]$ the basic element constructed by the procedure described in Proposition \ref{prop:differentialbasisinduction} for each $\sigma\in \Delta$. For the $n$-simplex $\sigma=\{v_0,\ldots,v_n\}\in \Delta$ with $v_0\prec\ldots\prec v_n$ there is just one sphere in the bouquet $|\Delta_{<\sigma}|$ as $K^\sigma_n=1$  (Example \ref{examplescnumberK}). It corresponds to the inclusions of simplices $v_n< \{v_{n-1},v_n\}<\ldots< \{v_1,\ldots,v_n\}<\{v_0,\ldots,v_n\}=\sigma$ (cf. Lemma \ref{lemmaposetwithlcfislocallyspherical}).  We claim and prove by induction that for $n\geq 1$:
\begin{equation}\label{equ:descriptionzsigmasimplicialcase}
z_\sigma=(-1)^n\sum_{i=0}^n(-1)^i s_{\partial_i(\sigma)}(z_{\partial_i(\sigma)}).
\end{equation}

Note that $\partial_i(\sigma)=\{v_0,\ldots,\widehat {v_i},\ldots,v_n\}$ for $i=0,\ldots,n$. For $n=1$ and $\sigma=\{v_0,v_1\}$ we have, according to Remark \ref{rmk:differentialbasisinductiondimq=1}, $z_\sigma=v_0-v_1=s_{v_0}(\emptyset_{v_0})-s_{v_1}(\emptyset_{v_1})=s_{\partial_1(\sigma)}(z_{\partial_1(\sigma)})-s_{\partial_0(\sigma)}(z_{\partial_0(\sigma)})$.  For the inductive step, consider $\sigma=\{v_0,\ldots,v_n\}$ with $n\geq 2$. To construct $Z=z_\sigma$ we must suspend $z=z_{\partial_0(\sigma)}=z_{\{v_1,\ldots,v_n\}}$ by means of two cones  $Z=Z_1+Z_2$, where $Z_1=(-1)^n s_{\partial_0(\sigma)}(z)$ and $Z_2$ is some element such that $\partial(Z_2)=z$. So it is enough to verify that the element $Z_2=(-1)^n\sum_{i=1}^n (-1)^i s_{\partial_i(\sigma)}(z_{\partial_i(\sigma)})$ satisfies $\partial(Z_2)=z$:
\begin{align*}
\partial(Z_2)&=(-1)^n\sum_{i=1}^n (-1)^i \partial(s_{\partial_i(\sigma)}(z_{\partial_i(\sigma)}))=\\
&=(-1)^n\sum_{i=1}^n (-1)^i \left(s_{\partial_i(\sigma)}(\partial(z_{\partial_i(\sigma)}))+(-1)^{n-1}z_{\partial_i(\sigma)}\right)=-\sum_{i=1}^n (-1)^i z_{\partial_i(\sigma)},
\end{align*}
where we have used Equation \eqref{equationsptppropertiesa} and that $\partial(z_{\partial_i(\sigma)})=0$ as $z_{\partial_i(\sigma)}$ is a cycle. Using the induction hypothesis we get
\[
\partial(Z_2)=(-1)^n\sum_{i=1}^n \sum_{j=0}^{n-1} (-1)^{i+j} s_{\partial_j(\partial_i(\sigma))}(z_{\partial_j(\partial_i(\sigma))}),
\]
which, by the simplicial identities, equals
\[
(-1)^{n-1}\sum_{k=0}^{n-1} (-1)^k s_{\partial_{k-1}(\partial_0(\sigma))}(z_{\partial_{k-1}(\partial_0(\sigma))}). 
\]
By the induction hypothesis this expression is exactly $z_{\partial_0(\sigma)}=z$. Now, from the description of $z_\sigma$ in \eqref{equ:descriptionzsigmasimplicialcase} and Equations \eqref{equ:dPq} and  \eqref{equationsptppropertiesc}, it is clear that
\[
d^{\P(\Delta)}_\sigma(b_\sigma)=\oplus_{i=0}^n(-1)^i b_{\partial_i(\sigma)},
\]
i.e., we recover the simplicial differential. This fact shows that Theorem \ref{thm_introd_ccsimp} of the introduction holds, see also Example \ref{examplescnumberK}. To illustrate the construction we reproduce it for $\sigma=\{0,1,2\}\in \Delta$ with $0\prec 1\prec 2$. We write simplices by juxtaposition for brevity, for instance, $\sigma=012$. Recall that we want to suspend the basic element $z=z_{12}=1-2$. The algorithm gives:
\[
Z_1=1<12-2<12\text{ and }Z_2=0<01-1<01-0<02+2<02.
\]
Thus for $Z=z_{012}=Z_1+Z_2$ we have:
\[
\partial(Z)=\partial_0(Z)-\partial_1(Z)=12-12+01-01-02+02-(1-2+0-1-0+2)=0.
\]
Moreover, the truncations
\[
t_{12}(Z)=1-2\text{, }t_{02}(Z)=2-0\text{ and }t_{01}(Z)=0-1, 
\]
give, by Equation \eqref{equ:dPq}, that:
\[
d^{\P(\Delta)}_\sigma([Z])=[t_{12}(Z)]\oplus [t_{02}(Z)]\oplus [t_{01}(Z)]=(1-2)\oplus(2-0)\oplus(0-1)=z_{12}\oplus (-z_{02})\oplus z_{01},
\]
and, by Equation \eqref{equ:dKq}, that:
\[
d^{\K_\Delta}_\sigma([Z])=[t_{12}(Z)]=1-2=z=z_{12}.
\]
\end{exa}

\begin{exa}\label{exampleQcdifferential}
For Quillen's complex $\A_p(G)$ of a finite group $G$ at a prime $p$ consider the local covering family $\K_{\A_p(G)}$ of Example \ref{exampleQccoveringfamilyK}. For a subgroup $H\cong C_p^{n+1}$ of dimension $n$ we have  $K^H_n=p^\frac{n(n+1)}{2}$ $(n-1)$-spheres in $|\A_p(G)_{<H}|$ (Example \ref{exampleQcnumberK}). Let $\prec$ be the chosen total order for the rank $1$ subgroups of $G$ and for any subgroup $W\leq H$ set $W^*=\min\{V|\text{$\dim(V)=0$ and $V\leq W$}\}$. Then the spheres in $|\A_p(G)_{<H}|$ are indexed by chain of subgroups $\rho=\{W_0<W_1<\ldots<W_{n-1}<W_n=H\}$, where $W_i$ does not contain $W_{i+1}^*$ for $i=0,\ldots,n-1$ (cf. Lemma \ref{lemmaposetwithlcfislocallyspherical}). It is clear that the chain $\rho$ can be written as 
\[
\rho=\{W_0=W_0^*<W_0^*\times W_1^*<\ldots<W_0^*\times W_1^*\times \ldots \times W_{n-1}^*\times W_n^*=H\}
\]
and that
\[
W_n^*\prec W_{n-1}^*\prec \ldots\prec W_1^*\prec W_0^*.
\]
For each $i\in \{0,\ldots,n\}$, we may define the following chain, constructed forgetting the subgroup $W_i^*$:
\[
\rho_i=\{W_0=W_0^*<W_0^*\times W_1^*<\ldots<W_0^*\times\ldots\times W_{i-1}^*\times W_{i+1}^*\times\ldots\times W_n^*=H_i\}.\]

Note that $H_i\cong C_p^{n}$ and $H_i\leq H$ for all $i$. Furthermore, the subposet of $\A_p(G)_{\leq H}$ with objects the subgroups $W^*_{i_0}\times W^*_{i_1}\times \ldots \times W^*_{i_l}$ with $0\leq i_0<\ldots<i_l\leq n$ and $0\leq l \leq n$ is isomorphic to the poset of non-empty subsets of the set $\{0,1,\ldots,n\}$ under inclusion.  Denote by $b_\rho=[z_\rho]$ and $b_{\rho_i}=[z_{\rho_i}]$ the basic elements associated to the chains $\rho$ and $\rho_i$ respectively by the procedure of Proposition \ref{prop:differentialbasisinduction}. Then the same argument as in Example \ref{examplescdifferential} shows that
\[
z_\rho=(-1)^n\sum_{i=0}^n (-1)^i s_{H_i}( z_{\rho_i})
\]
and that
\[
d^{\A_p(G)}_{H}(b_\rho)=\oplus_{i=0}^n (-1)^i b_{\rho_i}.
\]
This means that the differential of $C^{\K_{\A_p(G)}}_*(\A_p(G);R)$ behaves locally as the simplicial differential.

\end{exa}

\section{Free objects}\label{section:freeobjects}

The next result proves Theorem  \ref{thm_intro_freefaceonp} of the introduction:

\begin{thm}\label{thm_freefaces}
Let $\P$ be a locally $p$-Quillen finite poset at the prime $p$ with $\dim \P=N$. If the homology group $\tilde H_N(|\P|;R)$ is zero then proportion $r$ of free objects among the non-maximal objects of $P_{N-1}$ satisfies:
\[
r\geq \frac{p^{N+1}-2p^N+1}{p^{N+1}-p^N}> 0.
\]

\end{thm}
\begin{proof}
By Theorem \ref{chaincomplexforposetwithlcf} there is a chain complex 
\[
\xymatrix@C=10pt{
0\ar[r]&R^{|\P_N|p^\frac{N(N+1)}{2}}\ar[r]^{d_N}& R^{|\P_{N-1}|p^\frac{N(N-1)}{2}}\ar[r]&\ldots\ar[r]& R^{|\P_1|p}\ar[r]& R^{|\P_0|}\ar[r]&R\ar[r]&0
}
\]
whose homology is $\tilde H_*(|\P|;R)$. Write $\P_{N-1}=\P'\cup \P''$, where the objects of $\P'$ are maximal in $\P$ and the objects in $\P''$ are not. So, if for $p\in \P_{N-1}$ we set $n_p$ to be the number of maximal subgroups $q$ of $\P_N$ with $q>p$, we have $n_p=0$ for $p\in \P'$ and $n_p\geq 1$ for $p\in \P''$.  As $\tilde H_N(|\P|;R)=0$ the differential $d_N$ must be injective. In particular, we must have the inequality 
\begin{equation}\label{equationcollapsabilityinj}
|\P_N|p^\frac{N(N+1)}{2}\leq |\P''|p^\frac{(N-1)N}{2}.
\end{equation}
The number of edges in $\P$ among objects of $\P_N$ and objects of $\P_{N-1}$ is
\begin{equation}\label{equationcollapsabilityedges}
|\P_N|\frac{p^{N+1}-1}{p-1}=\sum_{p\in \P_{N-1}} n_p=\sum_{k=1}^K k|N_k|,
\end{equation}
where $N_k=\{p\in \P''|n_p=k\}$ and $K=\max\{n_p|p\in \P''\}$. As $|\P''|=\sum_{k=1}^K |N_k|$ we get from Equations \eqref{equationcollapsabilityinj} and \eqref{equationcollapsabilityedges} that
\begin{equation}\label{equationcollapsabilitymainineq}
\sum_{k=1}^K k|N_k|\leq f(p,N)(\sum_{k=1}^K |N_k|),
\end{equation}
where $f(p,N)=\frac{p^{N+1}-1}{p^N(p-1)}$. Note that $\lim_{N\to \infty} f(p,N)=\frac{p}{p-1}$ and $\lim_{p\to\infty} f(p,N)=1^+$. Now using the inequalities $k|N_k|>2|N_k|$ for $k\geq 3$ we obtain from Equation \eqref{equationcollapsabilitymainineq} that
\[
\frac{|N_1|}{|\P''|}\geq g(p,N),
\]
with $g(p,N)=2-f(p,N)=\frac{p^{N+1}-2p^N+1}{p^N(p-1)}$. Note that $\lim_{N\to \infty} g(p,N)=\frac{p-2}{p-1}$ and $\lim_{p\to\infty} g(p,N)=1^-$. The number $\frac{|N_1|}{|\P''|}$ is the ratio $r$ in the statement.

\end{proof}

\end{document}